\documentclass[12pt]{amsart}
\usepackage[margin=1in]{geometry}
\usepackage[foot]{amsaddr}
\usepackage{bbm}
\usepackage[numbers]{natbib}
\usepackage[main=english]{babel}
\usepackage[utf8]{inputenc}
\usepackage[T2A,T1]{fontenc}
\usepackage{graphicx}
\usepackage{mathrsfs}
\usepackage{mathtools}
\usepackage{titleps}
\usepackage{xcolor}
\definecolor{links}{rgb}{200,0,200}

\numberwithin{equation}{section}
\numberwithin{table}{section}
\numberwithin{figure}{section}

\usepackage{thmtools}

\theoremstyle{plain}
\newtheorem{theorem}{Theorem}[section]
\newtheorem*{theorem*}{Theorem}
\newtheorem{lemma}[theorem]{Lemma}

\theoremstyle{definition}

\theoremstyle{remark}
\newtheorem*{remark}{Remark}

\renewcommand{\tilde}{\widetilde}

\renewcommand{\H}{\mathbb H}

\newcommand{\D}{\mathbb D}
\newcommand{\R}{\mathbb{R}}

\newcommand{\eps}{\varepsilon}

\newcommand{\II}{\text{I\!I}}
\newcommand{\oploc}{\text{{\rm loc}}}
\newcommand{\opbdd}{\text{{\rm bdd}}}
\newcommand{\opII}{\text{{\rm II}}}
\newcommand{\opTr}{\text{{\rm Tr}}}
\newcommand{\opDet}{\text{{\rm Det}}}
\newcommand{\opLn}{\text{{\rm Ln}}}
\newcommand{\opaf}{\text{{\rm af}}}
\newcommand{\opExp}{\text{{\rm Exp}}}
\newcommand{\opMax}{\text{{\rm Max}}}
\def\opPSL{{\text{\rm PSL}}}
\catcode`\@=11
\def\eqalign#1{\null\,\vcenter{\openup1\jot \m@th %
\ialign{\strut\hfil$\displaystyle{{}##}$&$\displaystyle{{}##}$\hfil\crcr#1\crcr}}\,}
\def\triplealign#1{\null\,\vcenter{\openup1\jot \m@th %
\ialign{\strut\hfil$\displaystyle{##}$&$\displaystyle{{}##}$\hfil&$\displaystyle{{}##}$\hfil\crcr#1\crcr}}\,}
\def\multiline#1{\null\,\vcenter{\openup1\jot \m@th %
\ialign{\strut$\displaystyle{##}$\hfil&$\displaystyle{{}##}$\hfil\crcr#1\crcr}}\,}
\catcode`\@=12

\title{On a convexity property of the space of almost fuchsian immersions}
\date{\today}
\author{Samuel Bronstein, Graham Andrew Smith}
\begin{document}
\begin{abstract}
We study the space of Hopf differentials of almost fuchsian minimal immersions of compact Riemann surfaces. We show that the extrinsic curvature of the immersion at any given point is a concave function of the Hopf differential. As a consequence, we show that the set of all such Hopf differentials is a convex subset of the space of holomorphic quadratic differentials of the surface. In addition, we address the non-equivariant case, and obtain lower and upper bounds for the size of this set.
\end{abstract}

\maketitle
\tableofcontents
\section{Background and main result}

Convexity is always a noteworthy property whenever it is encountered. In this note, we describe a new convexity property of the space of almost fuchsian homomorphisms.

Let $\Bbb{H}^3$ denote $3$-dimensional hyperbolic space, let $\partial_\infty\Bbb{H}^3$ denote its ideal boundary, which we recall identifies with the Riemann sphere $\hat{\Bbb{C}}$, and let $\opPSL(2,\Bbb{C})$ denote its group of orientation-preserving isometries. Let $\rho$ be a discrete, injective homomorphism from a compact surface group $\Gamma$ into $\opPSL(2,\Bbb{C})$. This homomorphism is said to be {\sl fuchsian} whenever it preserves a circle in $\partial_\infty\Bbb{H}^3$, and {\sl quasi-fuchsian} whenever it preserves a Jordan curve, in which case the preserved Jordan curve is known to be unique.

The set of almost fuchsian homomorphisms was introduced by Uhlenbeck in \cite{Uhl83}. The homomorphism $\rho$ is said to be {\sl almost fuchsian} whenever, in addition to being quasi-fuchsian, its fixed Jordan curve bounds an embedded minimal disk in $\Bbb{H}^3$ with sectional curvatures of absolute values everywhere less than $1$. When this holds, the minimal disk is also unique. Almost fuchsian homomorphisms form a neighbourhood of the set of fuchsian homomorphisms with various interesting properties.

In order to state our result, we now alter our perspective slightly. Let $\mathbb{D}$ denote the Poincar\'e disk, and let $\opPSL(2,\mathbb{R})$ denote its group of orientation-preserving isometries. By Riemann's uniformization theorem, we may associate to every almost fuchsian homomorphism $\rho$ a pair $(\sigma,f)$, where $\sigma:\Gamma\rightarrow\opPSL(2,\mathbb{R})$ is a fuchsian homomorphism, and $f:\mathbb{D}\rightarrow\mathbb{H}^3$ is a $(\rho,\sigma)$-equivariant conformal minimal embedding. Furthermore $\sigma$ is unique up to conjugation and, once $\sigma$ is fixed, $f$ is unique. Conversely, the pair $(\sigma,f)$ uniquely determines $\rho$.

We call $\sigma$ the {\sl conformal type} of the almost fuchsian homomorphism, and henceforth suppose that it is fixed. Recall now that the Hopf differential of $f$ is defined to be the $(2,0)$-component of its shape operator. This is a $\sigma$-invariant holomorphic quadratic differential which determines $f$ up to postcomposition by isometries of $\mathbb{H}^3$, and which in turn determines $\rho$ up to conjugation.

Let $\mathcal{Q}(\Bbb{D},\sigma)$ denote the set of $\sigma$-invariant holomorphic quadratic differentials over $\Bbb{D}$, and let $\mathcal{Q}_\opaf(\Bbb{D},\sigma)$ denote the subset consisting of Hopf differentials of equivariant, almost fuchsian conformal minimal embeddings. By the preceding discussion, $\mathcal{Q}_\opaf(\Bbb{D},\sigma)$ parametrizes the space of conjugacy classes of almost fuchsian homomorphisms with conformal type $\sigma$.

\begin{theorem}\label{intro:thm:convexityI}
For any fuchsian homomorphism $\sigma:\Gamma\rightarrow\opPSL(2,\mathbb{R})$, $\mathcal{Q}_\opaf(\Bbb{D},\sigma)$ is a convex subset of $\mathcal{Q}(\Bbb{D},\sigma)$.
\end{theorem}

\begin{remark}
In \cite{Tra19}, it is shown that this subset is star-shaped about the origin.
\end{remark}

In fact, we prove a stronger result, from which Theorem \ref{intro:thm:convexityI} will immediately follow. In Section \ref{sec:AlmostFuchsian}, we provide a natural extension of the concept of almost fuchsian to the case of non-equivariant minimal embeddings. We denote by $\mathcal{Q}_\opbdd(\mathbb{D})$ the space of holomorphic quadratic differentials over $\mathbb{D}$ which are bounded with respect to the Poincar\'e metric. We furnish this space with the $C^0$ norm, and note that this makes it into a Banach space. We will see that the Hopf differential of every almost fuchsian conformal minimal embedding is an element of $\mathcal{Q}_\opbdd(\mathbb{D})$ and, furthermore, that every element of $\mathcal{Q}_\opbdd(\mathbb{D})$ is the Hopf differential of at most one such embedding.

Let $\mathcal{Q}_\opaf(\mathbb{D})\subseteq\mathcal{Q}_\opbdd(\mathbb{D})$ denote the set of Hopf differentials of almost fuchsian conformal minimal embeddings. For all $\phi\in\mathcal{Q}_\opaf(\mathbb{D})$, let $\kappa(\phi):\mathbb{D}\rightarrow\Bbb{R}$ denote the extrinsic curvature function of its corresponding almost fuchsian conformal minimal embedding.

\begin{theorem}\label{intro:thm:convexityIII}
For all $z\in\mathbb{D}$, the function
\begin{equation}
\kappa_z:\mathcal{Q}_\opaf(\mathbb{D})\rightarrow\mathbb{R};\phi\mapsto\kappa(\phi)(z)
\end{equation}
is concave.
\end{theorem}

\begin{remark}
Theorem \ref{intro:thm:convexityIII} is proven in Theorem \ref{thm:mainresultA}.
\end{remark}

\noindent As a consequence of Theorem \ref{intro:thm:convexityIII}, we obtain the following result.

\begin{theorem}\label{intro:thm:convexityII}
$\mathcal{Q}_\opaf(\mathbb{D})$ is a convex open subset of $\mathcal{Q}_\opbdd(\mathbb{D})$. Furthermore
\begin{equation}
B(0,1/2) \subseteq \mathcal{Q}_\opaf(\mathbb{D}) \subseteq B(0,1),
\end{equation}
where, for all $r>0$, $B(0,r)$ denotes the open ball of radius $r$ in $\mathcal{Q}_\opbdd(\mathbb{D})$.
\end{theorem}

\begin{remark}
A stronger version of Theorem \ref{intro:thm:convexityII} is proven in Theorem \ref{thm:mainresultB}.
\end{remark}

Our work also has the following interesting application to the study of Finsler metrics over Teichm\"uller space. Let $\mathcal{T}(\Gamma)$ denote the Teichm\"uller space of conjugacy classes of Fuchsian representations $\rho:\Gamma\rightarrow\opPSL(2,\mathbb{R})$. Recall that the cotangent space of $\mathcal{T}(\Gamma)$ at any point $\sigma$ identifies with the space $\mathcal{Q}(\mathbb{D},\sigma)$ of $\sigma$-invariant holomorphic quadratic differentials over $\mathbb{D}$. Since $\mathcal{Q}_\opaf(\mathbb{D},\sigma)$ is a convex neighbourhood of the origin in $\mathcal{Q}(\mathbb{D},\sigma)$, it thus defines a Finsler metric over $\mathcal{T}(\Gamma)$. More generally, in Sections \ref{sec:AlmostFuchsian} and \ref{sec:MainResult}, we construct a family $(\mathcal{U}_\lambda)_{\lambda\in(0,1]}$ of convex neighbourhoods of the origin in $\mathcal{Q}(\mathbb{D},\sigma)$ which, suitably rescaled, interpolates between the unit $L^\infty$ ball and $\mathcal{Q}_\opaf(\mathbb{D},\sigma)$. Bearing in mind the estimates derived in Theorem \ref{thm:mainresultB}, we thus have the following result.

\begin{theorem}
There exists a continuous family $(p_\lambda)_{\lambda\in[0,1]}$ of uniformly equivalent Finsler metrics $p$ on the Teichmüller space $\mathcal T_g$ such that $p_0$ coincides with the Finsler metric dual to the unit $L^{\infty}$ balls under the identification of the cotangent bundle with the space of holomorphic quadratic differentials.
\end{theorem}

\noindent The authors thank Nicolas Tholozan for his insights and discussions on this topic.

\section{Preliminaries}\label{sec:GaussCodazzi}

We first recall some elementary results of Gauss-Codazzi theory that will be used in the sequel. Let $h$ denote the Poincar\'e metric of $\mathbb{D}$, and let $f:\Bbb{D}\rightarrow\Bbb{H}^3$ be a smooth immersion which is conformal in the sense that its induced metric $g$ satisfies
\begin{equation}\label{eqn:ConformalMetric}
g = e^{2u}h\ ,
\end{equation}
for some smooth function $u$, which we call its \emph{conformal factor}. Let $\kappa$ denote the curvature of $g$, and recall that
\begin{equation}\label{eqn:CurvatureEquationI}
\Delta^h u = -1 - e^{2u}\kappa\ ,
\end{equation}
or, equivalently
\begin{equation}\label{eqn:CurvaturaEquationII}
\Delta^g u = -\kappa - e^{-2u}\ .
\end{equation}

Let $\opII$ denote the second fundamental form of $f$, let $A$ denote its shape operator, and let $H:=\opTr(A)/2$ denote its mean curvature. Recall that the \emph{Hopf differential} $\phi$ of $f$ is defined to be the $(2,0)$-component of $\opII$ with respect to $g$. In other words, it is the symmetric bilinear form given by
\begin{equation}\label{eqn:HopfDifferential}
\phi(\xi,\nu):=\frac{1}{4}\bigg(\II(\xi,\nu)-i\II(J\xi,\nu)-i\II(\xi,J\nu)-\II(J\xi,J\nu)\bigg)\ ,
\end{equation}
where $J$ here denotes the complex structure of $\mathbb{D}$. In particular,
\begin{equation}\label{eqn:gNormOfHopf}
\left|\phi\right|_g^2 = H^2 - \opDet(A).
\end{equation}
By Gauss' equation,
\begin{equation}\label{eqn:GaussEquation}
\opDet(A) = \kappa+1\ .
\end{equation}
Combined with \eqref{eqn:ConformalMetric}, \eqref{eqn:CurvatureEquationI} and \eqref{eqn:gNormOfHopf}, this yields
\begin{equation}\label{eqn:PreFundamentalEquation}
\Delta^h u = -1 + e^{2u}(1-H^2) + e^{-2u}\left|\phi\right|_h^2\ .
\end{equation}
We call the triple $(u,H,\phi)$ the \emph{Gauss--Codazzi} data of the immersion $f$. We recall the fundamental theorem of surface theory, as it applies in this context.

\begin{theorem}[Fundamental theorem of surface theory]\label{thm:FundamentalTheoremOfSurfaces}
For any triple $(u,H,\phi)$ satisfying \eqref{eqn:PreFundamentalEquation}, there exists a conformal immersion $f:\mathbb{D}\rightarrow\mathbb{H}^3$ with this Gauss--Codazzi data. Furthermore, $f$ is unique up to post-composition by isometries of $\mathbb{H}^3$.
\end{theorem}

Finally, we will make use of the following result of Li--Mochizuki \cite{LM20}.
\begin{theorem}[Sub-supersolution method]\label{subsup}
Let $(\Sigma,g)$ be a riemannian surface, let $F:\Sigma\times\Bbb{R}\rightarrow\Bbb{R}$ be a smooth function, and consider the equation
\begin{equation}\label{eqn:LiMochizuki}
\Delta^g u= F(x,u).
\end{equation}
If there exist $C^2$ functions $u_-,u_+:\Sigma\rightarrow\R$ satisfying
\begin{equation}
u_-<u_+\ ,\Delta^g u_-\geq F(x,u_-)\ ,\ \text{and}\ \Delta^g u_+\leq F(x,u_+)\ ,
\end{equation}
then \eqref{eqn:LiMochizuki} admits at least one solution $u$ satisfying $u_-\leq u\leq u_+$.
\end{theorem}

\noindent We will also require the following version of the Omori-Yau maximum principle (see \cite{Omo67, Yau75, LM20}).

\begin{theorem}[Omori-Yau maximum principle]\label{thm:OmoriYau}
Let $(M,g)$ be a complete riemannian manifold with Ricci curvature bounded below, and let $u:M\rightarrow\R$ be a $C^2$ function bounded above. For all $\eps>0$, there exists $x\in M$ such that
\begin{equation}
u(x)\geq \sup_{x\in M} u(x)-\eps\ ,\ |(\nabla^g u)(x)|\leq\eps\ ,\ \text{and}\ (\Delta^g u)(x)\leq\eps\ .
\end{equation}
\end{theorem}

\section{Almost fuchsian conformal minimal immersions}\label{sec:AlmostFuchsian}

We will be concerned with conformal immersions that are minimal in the sense that their mean curvature $H$ vanishes. In this case, the Hopf differential $\phi$ is known to be holomorphic (see [Hop50],\textsection 2), and \eqref{eqn:PreFundamentalEquation} clearly reduces to the fundamental identity
\begin{equation}\label{eqn:FundamentalIdentity}
\Delta^h u = -1 + e^{2u} + e^{-2u}\left|\phi\right|^2_h\ .
\end{equation}
It is this equation that will be the object of our study.

For all $(k,\alpha)$, and for every $k$-times differentiable function $f:\Bbb{D}\rightarrow\Bbb{R}$, let $\|f\|_{C^{k,\alpha}}$ denote the $C^{k,\alpha}$-norm of $f$ with respect to the Poincar\'e metric, and let $C^{k,\alpha}(\Bbb{D})$ denote the Banach space of $k$-times differentiable functions over $\Bbb{D}$ with bounded $C^{k,\alpha}$-norm. Let $\mathcal{Q}_\opbdd(\Bbb{D})$ denote the space of bounded holomorphic quadratic differentials over $\Bbb{D}$ furnished with the $C^0$ norm, and note that this is also a Banach space. Consider the function
\begin{equation}\label{eqn:FundamentalFunctional}
\mathcal{F}:C^{2,\alpha}(\Bbb{D})\times\mathcal{Q}_\opbdd(\Bbb{D})\rightarrow C^{0,\alpha}(\Bbb{D})\ ;(u,\phi)\mapsto \Delta^h u + 1 - e^{2u} - e^{-2u}\left|\phi\right|_h^2.
\end{equation}
Note that, since this function is constructed as a finite combination of differentiations and compositions by smooth functions, it defines a smooth function between Banach spaces.

Let $\mathcal{Z}:=\mathcal{F}^{-1}$ denote the solution space of $\mathcal{F}$, and let $\pi:\mathcal{Z}\rightarrow\mathcal{Q}_\opbdd(\mathbb{D})$ denote projection onto the second factor. The following result tells us that $\mathcal{Z}$ identifies with the space of complete conformal minimal immersions of $\mathbb{D}$ in $\mathbb{H}^3$ up to post-composition by isometries.

\begin{lemma}\label{ZIsConformalMinimalImmersions}
$\mathcal{Z}$ is the set of all Gauss--Codazzi data of complete conformal minimal immersions.
\end{lemma}

\begin{proof}
Indeed, let $f:\mathbb{D}\rightarrow\mathbb{H}$ be a conformal minimal immersion with conformal factor $u$, induced metric $g$, and Hopf differential $\phi$. Note that $g$ is complete, and let $\kappa$ denote its curvature. Consider now the equation
\begin{equation*}
\Delta^g v = -\kappa - e^{-2v}.
\end{equation*}
Applying Theorem \ref{subsup} with subsolution $v_-:=0$ and supersolution $v_+:=\opLn(2)/2$ yields a smooth solution $v$ of this equation satisfying
\begin{equation*}
0\leq v\leq\frac{\opLn(2)}{2}.
\end{equation*}
Since $g$ is complete, and since $v$ is bounded, $e^{2v}g$ is also complete. Since this metric has constant sectional curvature equal to $-1$, it follows by uniqueness of the Poincar\'e metric over $\mathbb{D}$ that
\begin{equation*}
e^{2v}g = h = e^{-2u}g,
\end{equation*}
so that $u=-v$ is also bounded. It now follows by quasi-linearity of \eqref{eqn:CurvatureEquationI} and the classical theory of elliptic operators that $u\in C^{2,\alpha}(\mathbb{D})$, so that $(u,\phi)\in\mathcal{Z}$, as desired. Conversely, if $(u,\phi)\in\mathcal{Z}$, then, by the fundamental theorem of surface theory, there exists a conformal minimal immersion $f:\mathbb{D}\rightarrow\mathbb{H}^3$ with this Gauss--Codazzi data. Since $u$ is bounded, this immersion is complete, and this completes the proof.
\end{proof}

The study of existence and uniqueness of solutions of $\mathcal{F}$ is equivalent to the study of the geometry of the pair $(\mathcal{Z},\pi)$. Indeed, the standard PDE problems of determining a priori estimates, existence and uniqueness correspond to the respective geometric problems of verifying properness, surjectivity and injectivity. In what follows, we study the extent to which $\mathcal{Z}$ is a graph over the image of $\pi$. To this end, we introduce, for all $\lambda>0$, the subsets $\Omega_\lambda\subseteq C^{2,\alpha}(\mathbb{D})\times\mathcal{Q}_\opbdd(\mathbb{D})$ and $\mathcal{U}_\lambda\subseteq\mathcal{Q}_\opbdd(\mathbb{D})$ given by
\begin{equation}\label{eqn:OmegaLambda}
\eqalign{
\Omega_\lambda&:=\big\{(u,\phi)\ \big|\ \sup_{z\in\mathbb{D}}e^{-4u(z)}\left|\phi(z)\right|^2_h<\lambda^2\big\}\ ,\ \text{and}\cr
\mathcal{U}_\lambda&:=\pi(\mathcal{Z}\cap\Omega_\lambda)\ .\vphantom{\big(}\cr}
\end{equation}

The sets $\Omega_1$ and $\mathcal{U}_1$ are closely related to the space of almost fuchsian embeddings. We first extend the definition given in the introduction (c.f. \cite{KS07,Sep16}). We say that a conformal minimal immersion $f:\mathbb{D}\rightarrow\mathbb{H}^3$, with conformal factor $u$, induced metric $g$, shape operator $A$, and Hopf differential $\phi$, is \emph{almost fuchsian} whenever
\begin{enumerate}
\item $f$ is properly embedded,
\item $\partial_\infty f(\mathbb{D})$ is a Jordan curve in $\partial_\infty\mathbb{H}^3$, and
\item $\sup_{z\in\mathbb{D}}\|A(z)\|_g < 1$.
\end{enumerate}
We call $\partial_\infty f(\mathbb{D})$ the \emph{boundary curve} of $f$. Note that, by \eqref{eqn:ConformalMetric} and \eqref{eqn:gNormOfHopf},
\begin{equation}\label{eqn:BoundedCurvatureEquivalences}
\sup_{z\in\Bbb{D}}e^{-2u(z)}\left|\phi(z)\right|_h = \sup_{z\in\Bbb{D}}\left|\phi(z)\right|_g = \sup_{z\in\Bbb{D}}\|A(z)\|\ ,
\end{equation}
so that the third condition given above has various equivalent formulations.

\begin{theorem}
$\mathcal{Z}\cap\Omega_1$ is the set of all Gauss--Codazzi data of almost fuchsian conformal minimal embeddings. Furthermore, if $f$ is an almost fuchsian conformal minimal embedding with Gauss--Codazzi data $(u,\phi)$, then its boundary curve is a $K$-quasicircle, where
\begin{equation}
K\leq\sup_{z\in\Bbb{D}}\frac{1+\left|\phi\right|_g}{1-\left|\phi\right|_g}\ .
\end{equation}
\end{theorem}

\begin{remark}
This result follows readily form the original work \cite{Uhl83} of Uhlenbeck as well as the argument provided in the proof of Theorem $1.3$ of \cite{HW13}. We provide here the details for completeness.
\end{remark}

\begin{proof}
By \eqref{eqn:BoundedCurvatureEquivalences} and Lemma \ref{ZIsConformalMinimalImmersions}, the Gauss--Codazzi data of any almost fuchsian conformal minimal embedding is an element of $\mathcal{Z}\cap\Omega_1$. We now prove the converse. Choose $(u,\phi)\in\mathcal{Z}\cap\Omega_1$, and let $f:\mathbb{D}\rightarrow\mathbb{H}^3$ be a conformal minimal immersion given by the fundamental theorem of surface theory. We will show that $f$ is almost fuchsian. Indeed, let $g:=e^{2u}h$ denote its induced metric, let $\nu$ denote its unit normal vector field, let $A$ denote its shape operator, and note that, by \eqref{eqn:BoundedCurvatureEquivalences},
\begin{equation}\label{ProofShapeOperatorBounded}
\sup_{z\in\Bbb{D}}\|A(z)\|_g < 1\ .
\end{equation}
Now define $\tilde{f}:\mathbb{D}\times\mathbb{R}\rightarrow\mathbb{H}^3$ by
\begin{equation*}
\tilde{f}(x,t):=\opExp_{f(x)}(t\nu(x))\ .
\end{equation*}
where $\opExp$ here denotes the exponential map of $\H^3$. As in Lemma $2.1$ of \cite{HW13}, we see that the metric induced by $\tilde{f}$ over $\mathbb{D}\times\mathbb{R}$ is
\begin{equation*}
\tilde{g}(x,t):=g_t(x)\oplus dt^2\ ,
\end{equation*}
where, for all $t$,
\begin{equation*}
g_t(\cdot,\cdot)=g((\cosh(t)I_0+\sinh(t)A)\cdot,(\cosh(t)I_0+\sinh(t)A)\cdot)\ .
\end{equation*}
By \eqref{ProofShapeOperatorBounded}, $\tilde{g}$ is everywhere non-degenerate and, since $g$ is complete, so too is $\tilde{g}$. In particular, since $\tilde{f}$ is a local isometry with respect to this metric, it satisfies the path-lifting property. Since $\mathbb{H}^3$ is simply-connected, it follows that $\tilde{f}$ is a diffeomorphism, so that $f$ is indeed properly embedded.

We now show that $\partial_\infty f(\Bbb{D})$ is a Jordan curve in $\partial_\infty\mathbb{H}^3$. To see this, let $f':\D\rightarrow\H^3$ be an isometric parametrization of a totally geodesic plane, let $\nu'$ denote its unit normal vector field, and define $\tilde{f}':\D\times\R\rightarrow\H^3$ by
\begin{equation*}
\tilde{f}'(x,t)=\opExp_{f'(x)}(t\nu'(x))\ .
\end{equation*}
The metric induced by $\tilde f'$ is
\begin{equation*}
\tilde{g}'(x,t)=\cosh^2(t)g\oplus dt^2\ .\label{fuchsexp}
\end{equation*}
By \eqref{ProofShapeOperatorBounded}, $\alpha:=e\circ f^{-1}$ is bilipschitz. It therefore extends (see, for example, Gehring \cite{Geh62} or Mostow \cite{Mos68}) to a quasiconformal map $\tilde{\alpha}$ from $\partial_\infty\H^3$ to itself, from which it follows that $\partial_\infty f(\D)$ is a quasicircle. In particular, it is a Jordan curve, as asserted.

Finally, for all $t$, the restrictions of $\tilde{g}'$ and $\tilde{g}$ to the slice $\D\times\{t\}$ are $K$-quasiconformal to one-another, where
\begin{equation*}
K=\frac{1+\left|\phi\right|_g}{1-\left|\phi\right|_g}\ .
\end{equation*}
Letting $t$ tend to $+\infty$, we see that $\tilde \alpha$ is $K$-quasiconformal. Since $\partial_\infty f(\mathbb{D})$ is the image under $\tilde{\alpha}$ of an equator in $\partial_\infty\mathbb{H}^3$, it follows that this curve is a $K$-quasicircle, and this completes the proof.
\end{proof}

\section{Proofs of main results}\label{sec:MainResult}

We now prove the main results of this note. For all $\phi\in\mathcal{U}_1$, let $\kappa(\phi)$ denote the extrinsic curvature function of its corresponding almost fuchsian conformal minimal embedding.

\begin{theorem}\label{thm:mainresultA}
For all $z\in\mathbb{D}$, the function
\begin{equation}
\kappa_z:\mathcal{U}_1\rightarrow\Bbb{R};\phi\mapsto\kappa(\phi)(z)
\end{equation}
is concave.
\end{theorem}

\begin{theorem}\label{thm:mainresultB}
The solution space $\mathcal{Z}$ has the following properties.
\begin{enumerate}
\item $\mathcal{Z}\cap\Omega_1$ is a smooth graph over its image,
\item for all $0<\lambda\leq 1$, $\mathcal{U}_\lambda$ is convex, and
\item for all $0<\lambda\leq 1$,
\begin{equation}
B(0,\lambda/(1+\lambda^2)) \subseteq \mathcal{U}_\lambda \subseteq B(0,\lambda),
\end{equation}
where, for all $r>0$, $B(0,r)$ denotes the open ball of radius $r$ about the origin in $\mathcal{Q}(\mathbb{D})$ with respect to the $C^0$ norm.
\end{enumerate}
\end{theorem}

In order to prove these results, we first consider the more general function
\begin{equation}\label{eqn:DefOfHatF}
\hat{\mathcal{F}}:C^{2,\alpha}(\Bbb{D})\times C^{0,1}(\Bbb{D})\rightarrow C^{0,\alpha}(\Bbb{D})\ ;(u,f)\mapsto \Delta^h u + 1 - e^{2u} - e^{-2u}f.
\end{equation}
As before, $\hat{\mathcal{F}}$ is a smooth function between Banach spaces. Let $\hat{\mathcal{Z}}:=\hat{\mathcal{F}}^{-1}$ denote its zero set, let $\pi:Z\rightarrow C^{0,1}(\Bbb{D})$ denote the projection onto the first factor, and, for all $\lambda>0$, define
\begin{equation}\label{eqn:OmegaLambdaAgain}
\eqalign{
\hat{\Omega}_\lambda&:=\big\{(u,f)\ \big|\ \sup_{z\in\mathbb{D}}e^{-4u(z)}f(z)<\lambda^2\big\}\ ,\ \text{and}\cr
\hat{\mathcal{U}}_\lambda&:=\pi(\mathcal{Z}\cap\hat{\Omega}_\lambda)\ .\vphantom{\big(}\cr}
\end{equation}

\begin{lemma}\label{lemma:aProiriEstimates}
For all $\lambda>0$, and for all $(u,f)\in\hat{\mathcal{Z}}\cap\hat{\Omega}_\lambda$,
\begin{equation}\label{ref:aProiriEstimates}
-\frac{1}{2}\opLn(1+\lambda^2)\leq u\leq 0\ .
\end{equation}
\end{lemma}

\begin{proof} Suppose first that
\begin{equation*}
c^+:=\sup_{z\in\Bbb{D}}u(z)>0\ .
\end{equation*}
By the Omori-Yau maximum principle, there exists $z_0\in\Bbb{D}$ such that
\begin{equation*}
(\Delta^h u)(z_0) < e^{c^+} - 1\ ,\ \text{and}\ u(z_0) > c^+/2\ .
\end{equation*}
At this point,
\begin{equation*}
\hat{\mathcal{F}}(u,f) < (e^{c^+} - 1) + 1 - e^{c^+} = 0\ ,
\end{equation*}
which is absurd, and it follows that $u\leq 0$.

Suppose now that
\begin{equation*}
c^-:=\inf_{z\in\Bbb{D}}u(z)<-\frac{1}{2}\opLn(1+\lambda^2)\ .
\end{equation*}
Choose $\delta>0$ such that
\begin{equation*}\triplealign{
&c^-+\delta &< -\frac{1}{2}\opLn(1+\lambda^2)\cr
\Leftrightarrow&e^{2(c^-+\delta)}(1+\lambda^2) &< 1\ .\cr}
\end{equation*}
By the Omori-Yau maximum principle, there exists $z_0\in\Bbb{D}$ such that
\begin{equation*}
(\Delta^h u)(z_0) > -(1 - e^{2(c^-+\delta)}(1+\lambda^2))\ ,\ \text{and}\ u(z_0) < c^-+\delta\ .
\end{equation*}
At this point,
\begin{equation*}
e^{-2u(z_0)} f(z_0) = e^{2u(z_0)}(e^{-4u(z_0)}f(z_0)) \leq \lambda^2 e^{2u(z_0)}\ .
\end{equation*}
Consequently,
\begin{equation*}
\hat{\mathcal{F}}(u,f) > -(1 - e^{2(c^-+\delta)}(1+\lambda^2)) + 1 - e^{2u(z_0)}(1+\lambda^2)  > 0\ ,
\end{equation*}
which is absurd. It follows that $u\geq -\opLn(1+\lambda^2)/2$, and this completes the proof.
\end{proof}

\begin{lemma}\label{lemma:uniqueness}
The restriction of $\pi$ to $\hat{Z}\cap\hat{\Omega}_1$ is injective.
\end{lemma}

\begin{proof}Suppose the contrary, so that there exist $u,u\in C^{2,\alpha}(\Bbb{D})$ and $f\in C^{0,1}(\Bbb{D})$ such that $(u,f),(u',f)\in\hat{\mathcal{Z}}\cap\hat{\Omega}_1$. Denote $w:=u-u'$, and without loss of generality, suppose that
\begin{equation*}
c:=\sup_{x\in\Bbb{D}}w(z)>0\ .
\end{equation*}
Choose $0<\delta<1$ be such that
\begin{equation*}
\delta<\inf_{z\in\Bbb{D}} 2e^{2u'(z)}(1-e^{-4u'(z)}f(z))\ .
\end{equation*}
By hypothesis,
\begin{equation*}
\hat{\mathcal{F}}(u,f) = \hat{\mathcal{F}}(u',f)=0\ .
\end{equation*}
By convexity, for all $x,y\in\Bbb{R}$,
\begin{equation*}\label{eqn:convexity}
e^y-e^x \geq e^x (y-x)\ .
\end{equation*}
Bearing in mind \eqref{eqn:DefOfHatF}, these relations together yield,
\begin{equation*}
\Delta^h w = \Delta^h (u-u') = e^{2u} + e^{-2u}f - e^{2u'} - e^{-2u'}f\ .
\end{equation*}
When $w\geq 0$ and $f\geq 0$, by \eqref{eqn:convexity},
\begin{equation*}
\Delta^hw\geq2e^{2u'}(u-u') + 2e^{-2u'}(u'-u)f=2e^{2u'}(1-e^{-4u'}f)w\geq\delta w\ .
\end{equation*}
On the other hand, when $w\geq 0$ and $f<0$,
\begin{equation*}
\Delta^hw\geq2e^{2u'}(u-u')=2e^{2u'}w\geq\delta w\ .
\end{equation*}
However, by the Omori-Yau maximum principle, there exists $z_0\in\Bbb{D}$ such that $(\Delta^h w)(z_0)<(\delta c)/2$ and $w(z_0)>c/2$. This contradicts the above relation, and it follows that $w$ vanishes, as desired.
\end{proof}

\begin{lemma}\label{lemma:openness}
$\hat{\mathcal{Z}}\cap\hat{\Omega}_1$ is a Banach submanifold modelled on $C^{0,1}(\Bbb{D})$ and the canonical projection $\pi:\hat{\mathcal{Z}}\cap\hat{\Omega}_1\rightarrow\hat{\mathcal{U}}_1$ is diffeomorphism onto its image.
\end{lemma}

\begin{proof} The partial derivative of $\mathcal{F}$ with respect to the first component is
\begin{equation*}
L_{u,f} := D_1\hat{\mathcal{F}}(u,f)\cdot v = \Delta^h v - 2e^{2u}(1-e^{-4u}f)v\ .
\end{equation*}
We claim that, for all $(u,f)\in\hat{\mathcal{Z}}\cap\hat{\Omega}_1$, $L_{u,f}$ defines a linear isomorphism from $C^{2,\alpha}(\Bbb{D})$ into $C^{0,\alpha}(\Bbb{D})$. Let $\delta>0$ be such that, for all $z\in\Bbb{D}$,
\begin{equation*}
2e^{2u(z)}\big(1-e^{-4u(z)}f(z)\big) > \delta\ .
\end{equation*}

We first show that $L_{u,f}$ is injective. Indeed, let $v\in C^{2,\alpha}(\Bbb{D})$ be an element of its kernel. Suppose that
\begin{equation*}
c:=\sup_{z\in\Bbb{D}} v(z)>0\ .
\end{equation*}
By the Omori-Yau maximum principle, there exists $z_0\in\Bbb{D}$ such that
\begin{equation*}
(\Delta^h v)(z_0) < \frac{\delta c}{2},\ \text{and}\ v(z_0) > \frac{c}{2}\ .
\end{equation*}
At this point
\begin{equation*}
L_{u,f}v = \Delta^h v  - 2e^{2u}(1-e^{-4u}f)v < \frac{\delta c}{2}-\frac{\delta c}{2} = 0\ ,
\end{equation*}
which is absurd, and it follows that $v\leq 0$. In a similar manner, we show that $v\geq 0$, so that $v$ vanishes. $L_{u,f}$ therefore has trivial kernel in $C^{2,\alpha}(\Bbb{D})$ and is thus injective, as asserted.

We now prove surjectivity. Choose $w\in C^{0,\alpha}(\Bbb{D})$. Choose $z_0\in\Bbb{D}$ and, for all $r>0$, let $B_r(z_0)$ denote the open ball of hyperbolic radius $r$ about this point. Since the zero'th order coefficient of $L_{u,\phi}$ is negative, for all $r$, there exists $v_r\in C_0^{2,\alpha}(B_r(z_0))$ such that
\begin{equation*}
L_{u,f}v_r = w\ .
\end{equation*}
Furthermore, by the maximum principle, for all $r$,
\begin{equation*}
\|v_r\|_{C^0} \leq \frac{1}{\delta}\|w\|_{C^0}\ .
\end{equation*}
It follows by the Schauder estimates that there exists $C>0$ such that, for all $r$, and for all $s>r+1$,
\begin{equation*}
\|v_s\|_{C^{2,\alpha}(B_r(z_0))} \leq C\big(\|v_s\|_{C_0} + \|L_{u,\phi}v_s\|_{C^{0,\alpha}}\big) \leq C\bigg(1+\frac{1}{\delta}\|w\|_{C^{0,\alpha}}\bigg)\ .
\end{equation*}
It follows by the Arzela-Ascoli theorem that $(v_r)_{r>0}$ contains a subsequence that converges in the $C^{2,\beta}_\oploc$ sense, for all $\beta<\alpha$ to some $v\in C^{2,\alpha}(\Bbb{D})$ satisfying $L_{u,f}v=w$. This proves surjectivity.

Finally, by the closed graph theorem, $L_{u,f}$ is a linear isomorphism, and it follows by the implicit function theorem that $\hat{\mathcal{Z}}\cap\hat{\Omega}_1$ is a Banach manifold modelled on $C^{0,1}(\Bbb{D})$ and that the canonical projection is a local diffeomorphism onto its image. By Lemma \ref{lemma:uniqueness}, the canonical projection is then a diffeomorphism onto its image, and this completes the proof.
\end{proof}

\begin{lemma}\label{lemma:monotonicity}
If $f\in\hat{\mathcal{U}}_1$ and if $g\leq f$, then $f\in\hat{\mathcal{U}}_1$.
\end{lemma}

\begin{proof}
Indeed, define $f:\Bbb{D}\times[0,1]\rightarrow\Bbb{R}$ by
\begin{equation*}
f_t(z) := f(z,t) := (1-t)f + tg\ .
\end{equation*}
and note that $(f_t(z))_{t\in[0,1]}$ is non-increasing for all $z$. Let $I\subseteq[0,1]$ denote the set of all $t\in[0,1]$ with the property that, for all $s\in[0,t]$, $f_s\in\hat{\mathcal{U}}$. By hypothesis, $0\in I$ and, by Lemma \ref{lemma:openness}, $I$ is open. We now show that $I$ is closed. Indeed, let $t_0$ denote the supremum of $I$. By Lemma \label{Lemma:openness} there exists a smooth family $(u_t)_{t\in[0,t_0[}$ in $C^{2,\alpha}(\Bbb{D})$ such that, for all $t$, $(u_t,f_t)\in\hat{\Omega}_1$, and,
\begin{equation*}
\hat{\mathcal{F}}(u_t,f_t) = 0\ .
\end{equation*}
Note that this means that the time derivative $\dot{u}_t$ is also an element of $C^{2,\alpha}(\Bbb{D})$ for all $t$.

We first show that, for all $z\in\Bbb{D}$, $(e^{-4u_t(z)}f_t(z))_{t\in[0,t_0[}$ is non-increasing in $t$. Indeed, differentiating the identity $\hat{\mathcal{F}}(u_t,f_t)=0$ with respect to $t$ yields, for all $t$,
\begin{equation*}
\Delta^h\dot{u}_t=2e^{2u_t}(1-e^{-4u_t}f_t)\dot{u}_t + e^{-2u_t}(g-f)\ .
\end{equation*}
Since $(u_t)_{t\in[0,t_0[}$ varies smoothly in $C^{2,\alpha}(\Bbb{D})$ with $t$, $\dot{u}_t$ is also an element of $C^{2,\alpha}(\Bbb{D})$ for all $t$. Now choose $t\in[0,t_0[$, and denote
\begin{equation*}\eqalign{
\delta &:= \inf_{z\in\Bbb{D}}2e^{2u_t(z)}\big(1-e^{-4u_t(z)}f_t(z)\big)\ ,\ \text{and}\cr
c &:= \inf_{z\in\Bbb{D}}\dot{u}_t(z)\ .\cr}
\end{equation*}
We claim that $c\geq 0$. Indeed, suppose the contrary. By the Omori-Yau maximum principle, there exists $z_0\in\Bbb{D}$ such that
\begin{equation*}
(\Delta^h\dot{u}_t)(z_0) > \frac{-c\delta}{2},\ \text{and},\ u_t(z_0) < \frac{-c}{2}\ .
\end{equation*}
Since $g\leq f$, at this point
\begin{equation*}
\Delta^h\dot{u}_t-2e^{2u_t}(1-e^{-4u_t}f_t)\dot{u}_t-e^{-2u_t}(g-f) > \frac{-c\delta}{2} + \frac{c\delta}{2} = 0\ .
\end{equation*}
This is absurd, and it follows that $c\geq 0$ as asserted. In other words, for all $t$,
\begin{equation*}
\dot{u}_t\geq 0\ ,
\end{equation*}
so that $(u_t)_{t\in[0,t_0[}$ is non-decreasing in $t$, and it follows that $(e^{-4u_t(z)}f_t(z))_{t\in[0,t_0[}$ is non-increasing in $t$, as asserted.

We now prove a compactness property of the family $(u_t)_{t\in[0,t_0[}$. By Lemma \ref{lemma:aProiriEstimates}, there exists $C_1>0$ such that, for all $t\in[0,t_0[$,
\begin{equation*}
\|u_t\|_{C^0}<C_1\ .
\end{equation*}
Since $\hat{\mathcal{F}}$ is quasilinear, by the Schauder estimates, there exists $C_2>0$ such that, for all $t\in[0,t_0[$,
\begin{equation*}
\|u_t\|_{C^{2,\alpha}}<C_2\ .
\end{equation*}
It follows by the Arzela-Ascoli theorem $(u_t)_{t\in[0,t_0[}$ is a relatively compact subset of $C^{2,\alpha}(\Bbb{D})$ in the $C^{2,\beta}_\oploc$ topology, for all $\beta<\alpha$.

By compactness, the family $(u_t)_{t\in[0,t_0[}$ has a $C^{0,\beta}_\oploc(\Bbb{D})$-accumulation point $v\in C^{2,\alpha}(\Bbb{D})$ such that $\hat{\mathcal{F}}(v,f_{t_0})=0$. By monotonicity, $(v,f_{t_0})\in\hat{\Omega}_1$, so that $f_{t_0}\in\hat{\mathcal{U}}_1$. It follows that $t_0\in I$, and $I$ is therefore closed. Since $0$ is trivially an element of $I$, it follows by connectedness that so too is $1$, and this completes the proof.
\end{proof}

Note now that the function
\begin{equation*}
\phi\mapsto\left|\phi\right|_h^2
\end{equation*}
trivially maps $\mathcal{Q}_\opbdd(\Bbb{D})$ smoothly into $C^{0,1}(\Bbb{D})$. Lemma \ref{lemma:openness} thus immediately yields the following result.

\begin{lemma}\label{lemma:opennessB}
$\mathcal{Z}\cap\Omega_1$ is a Banach submanifold modelled on $\mathcal{Q}_\opbdd(\Bbb{D})$ and the canonical projection $\pi:\mathcal{Z}\cap\Omega_1\rightarrow\mathcal{U}_1$ is a diffeomorphism onto its image.
\end{lemma}

\noindent Likewise, Lemma \ref{lemma:monotonicity} immediately yields the following result.

\begin{lemma}\label{lemma:monotonicityB}
Let $\phi,\psi\in\mathcal{Q}_\opbdd(\Bbb{D})$ be holomorphic quadratic differentials. If $\phi\in\mathcal{U}_1$ and if, for all $z$,
\begin{equation}
\left|\psi(z)\right| \leq \left|\phi(z)\right|\ ,
\end{equation}
then $\psi\in\mathcal{U}_1$. In particular, $\mathcal{U}_1$ is star-shaped about the origin.
\end{lemma}

\noindent Finally, Lemma \ref{lemma:aProiriEstimates} yields the following lower and upper bounds for $\mathcal{U}_\lambda$, for all $0<\lambda\leq 1$.

\begin{lemma}\label{lemma:lowerandupperbound}
For all $0<\lambda\leq 1$,
\begin{equation}
\overline{B}(0,\lambda/(1+\lambda^2)) \subseteq \mathcal{U}_\lambda \subseteq B(0,\lambda),
\end{equation}
\end{lemma}

\begin{proof} Suppose first that $\phi\in\mathcal{U}_\lambda$, and let $u\in C^{2,\alpha}$ be such that $(u,\phi)\in\mathcal{Z}\cap\Omega_\lambda$. By \eqref{lemma:aProiriEstimates}, $u\leq0$, so that, for all $z\in\mathbb{D}$,
\begin{equation*}
\left|\phi\right|_h < e^{2u(z)}\lambda < \lambda\ .
\end{equation*}
It follows that
\begin{equation*}
\sup_{z\in\Bbb{D}}\left|\phi(z)\right|_h < \lambda\ ,
\end{equation*}
as desired.

Now suppose that
\begin{equation*}
\sup_{z\in\mathbb{D}}\left|\phi(z)\right|_h < \frac{\lambda}{1+\lambda^2}\ .
\end{equation*}
Upon applying Lemma \ref{subsup} with subsolution $u_-:=-\frac{1}{2}\opLn(1+\lambda^2)$ and $u_+:=0$, we see that there exists $u\in C^{2,\alpha}(\mathbb{D})$ such that $\mathcal{F}(u,\phi)=0$, and
\begin{equation*}
-\frac{1}{2}\opLn(1+\lambda^2)\leq u\leq 0\ .
\end{equation*}
In particular, for all $z\in\mathbb{D}$,
\begin{equation*}
e^{-2u(z)}\left|\phi(z)\right|_h < (1+\lambda^2)\cdot\frac{\lambda}{1+\lambda^2} = \lambda\ .
\end{equation*}
It follows that $(u,\phi)\in\mathcal{Z}\cap\Omega_\lambda^2$, so that $\phi\in\mathcal{U}_\lambda^2$, as desired.
\end{proof}

It now only remains to prove convexity.

\begin{proof}[Proof of Theorem \ref{thm:mainresultA}]Choose $\phi_0,\phi_1\in\mathcal{U}_1$ such that, for all $t\in[0,1]$,
\begin{equation*}
\phi_t := (1-t)\phi_0 + t\phi_1\in\mathcal{U}_1\ .
\end{equation*}
For all $t$, let $u_t\in C^{2,\alpha}(\Bbb{D})$ denote the unique function such that $(u_t,\phi_t)\in\mathcal{Z}\cap\Omega_1$, and let $\kappa_t:\mathbb{D}\rightarrow\mathbb{R}$ denote the extrinsic curvature function of its corresponding almost fuchsian conformal minimal immersion. Note that, by \eqref{eqn:BoundedCurvatureEquivalences}, for all $t$, and for all $z\in\mathbb{D}$,
\begin{equation*}
\kappa_t(z) = -e^{-4u_t(z)}\left|\phi_t(z)\right|_h^2.
\end{equation*}

We claim that
\begin{equation*}
\ddot{u}_t + 4\dot{u}_t^2\leq 0.
\end{equation*}
Indeed, differentiating the equation $\mathcal{F}(u_t,\phi_t)=0$ twice with respect to $t$ yields
\begin{equation}\label{cvx}
\eqalign{
\Delta\dot u_t&=2\dot u_t e^{2u_t} -2\dot u_te^{-2u_t}|\phi_t|^2+2e^{-2u_t}g(\phi_t,\phi_1-\phi_0)\ ,\ \text{and}\cr
\Delta\ddot u_t&=2\ddot u_t(e^{2u_t}-e^{-2u_t}|\phi_t|^2)+4\dot u_t^2(e^{2u_t}+e^{-2u_t}|\phi_t|^2)\cr
&\qquad\qquad-8\dot u_te^{-2u_t}g(\phi_t,\phi_1-\phi_0)+2e^{-2u_t}|\phi_1-\phi_0|^2\ .\cr}
\end{equation}
Note now that
\begin{equation*}
\Delta(\dot u_t^2)=2\dot u_t\Delta\dot u_t+2|\nabla\dot u_t|^2\geq 2\dot u_t\Delta\dot u_t\ .
\end{equation*}
Combining this with \eqref{cvx} yields
\begin{equation*}\eqalign{
\Delta(\ddot u_t+4\dot u_t^2)&\geq 2\ddot u_t(e^{2u_t}-e^{-2u_t}|\phi_t|^2)+20\dot u_t^2e^{2u_t}-12\dot u_t^2e^{-2u_t}|\phi_t|^2\cr
&\qquad\qquad+8\dot u_te^{-2u_t}g(\phi_t,\phi_1-\phi_0)+2e^{-2u_t}|\phi_1-\phi_0|^2\ .\cr}
\end{equation*}
Judiciously rearranging the terms on the right-hand side, we obtain
\begin{equation*}\eqalign{
\Delta(\ddot u_t+4\dot u_t^2)&\geq 2(\ddot u_t+4\dot u_t^2)(e^{2u_t}-e^{-2u_t}|\phi_t|^2)+12\dot u_t^2(e^{2u_t}-e^{-2u_t}|\phi_t|^2)\cr
&\qquad\qquad+8\dot u_t^2e^{-2u_t}|\phi_t|^2+8e^{-2u_t}g(\dot u_t\phi_t,\phi_1-\phi_0)+2e^{-2u_t}|\phi_1-\phi_0|^2\ .\cr}
\end{equation*}
Note now that, for all $a,b\in\Bbb{R}$,
\begin{equation*}
8a^2 + 8ab + 2b^2=2(2a+b)^2\geq 0\ .
\end{equation*}
so that
\begin{equation*}\eqalign{
\Delta(\ddot u_t+4\dot u_t^2)&\geq2e^{2u_t}(1-e^{-4u_t}|\phi_t|^2)(\ddot u_t+4\dot u_t^2)+12\dot u_t^2e^{2u_t}(1-e^{-4u_t}|\phi_t|^2)\cr
&\geq2e^{2u_t}(e^{2u_t}-e^{-2u_t}|\phi_t|^2)(\ddot u_t+4\dot u_t^2)\ ,\cr}
\end{equation*}
and it follows by the Omori-Yau maximum principle that
\begin{equation*}
\ddot u_t+4\dot u_t^2\leq 0\ ,
\end{equation*}
as asserted.

We now prove concavity. Indeed, for all $z\in\mathbb{D}$,
\begin{equation*}\eqalign{
e^{4u_t(z)}\partial^2_t\kappa_t(z)&=-e^{4u_t(z)}\partial^2_t(e^{-4u_t(z)}|\phi_t(z)|^2)\cr
&=4(\ddot u_t(z)+4\dot u_t^2(z))|\phi_t(z)|^2-32\dot u_t(z)^2|\phi_t|^2\cr
&\qquad\qquad+16\dot u_t(z)g(\phi_t(z),\phi_1(z)-\phi_0(z))-2|\phi_1(z)-\phi_0(z)|^2\cr
&=+4(\ddot u_t(z)+4\dot u_t(z)^2)|\phi_t(z)|^2 - 2|4\dot{u}_t(z)\phi_t(z) - (\phi_1(z)-\phi_0(z))|^2\cr
&\leq 0\ ,\cr}
\end{equation*}
as desired.
\end{proof}

We now prove Theorem \ref{thm:mainresultA}.

\begin{proof}[Proof of Theorem \ref{thm:mainresultA}]
Items $(1)$ and $(3)$ follow by Lemmas \ref{lemma:opennessB} and \ref{lemma:lowerandupperbound} respectively. By Lemma \ref{lemma:monotonicity}, for all $\lambda\in[0,1]$, $\mathcal{U}_\lambda$ is star-shaped. We now prove convexity. Let $\lambda\in[0,1]$, $\phi_0,\phi_1\in\mathcal{U}_\lambda$, and $\epsilon>0$ be such that $\phi_0,\phi_1\in\mathcal{U}_{\lambda-\epsilon}$. Consider the family
\begin{equation*}
\phi_{s,t}:=s(1-t)\phi_0 + st\phi_1\ .
\end{equation*}
Let $I\subseteq[0,1]$ denote the set of all $s\in[0,1]$ for which $\phi_{s,t}\in\mathcal{U}_\lambda$ for all $t\in[0,1]$, and, for all $(s,t)\in I\times[0,1]$, let $u_{s,t}\in C^{2,\alpha}(\mathbb{D})$ denote the unique function such that $(u_{s,t},\phi_{s,t})\in\mathcal{Z}\cap\Omega_\lambda$. We use a connectedness argument to show that $I$ coincides with $[0,1]$. Indeed, trivially, $0\in I$ and, by Lemma \ref{lemma:opennessB}, $I$ is open. We now show that $I$ is also closed. Indeed, let $(s_m,t_m)_{m\in\Bbb{N}}$ be a sequence in $I\times[0,1]$ converging to $(s_\infty,t_\infty)$, say, and, for all $m$, denote $u_m:=u_{s_m,t_m}$. By Lemma \ref{lemma:aProiriEstimates}, for all $m$,
\begin{equation*}
-\frac{1}{2}\opLn(1+\lambda^2)\leq u_m\leq 0\ .
\end{equation*}
Since $\mathcal{F}$ is quasi-linear, it follows by the classical theory of elliptic operators that $(u_m)_{m\in\Bbb{N}}$ is precompact in $C^{2,\alpha}(\Bbb{D})$ with respect to the $C^{2,\alpha}_\oploc$ topology. Let $u_\infty\in C^{2,\alpha}(\Bbb{D})$ be an accumulation point. We claim that $(u_\infty,\phi_{s_\infty,t_\infty})\in\mathcal{Z}\cap\Omega_\lambda$. Indeed, by convexity, for all $z\in\mathbb{D}$, and for all $m$,
\begin{equation*}
e^{-4u_m(z)}\left|\phi_{s_m,t_m}(z)\right|_h^2 < \opMax\big(e^{-4u_{s_m,0}(z)}\left|\phi_{s_m,0}(z)\right|_h^2,e^{-4u_{s_m,1}(z)}\left|\phi_{s_m,1}(z)\right|_h^2\big)\leq(\lambda-\epsilon)^2\ .
\end{equation*}
Upon taking limits, it follows that, for all $z\in\mathbb{D}$,
\begin{equation*}
e^{-4u_\infty(z)}\left|\phi_{s_\infty,t_\infty}(z)\right|_h^2 \leq(\lambda-\epsilon)^2\ ,
\end{equation*}
so that $\phi_{s_\infty,t_\infty}\in\mathcal{U}_\lambda$. Since $(s_m,t_m)_{m\in\Bbb{N}}$ is arbitrary, it follows that $I$ is indeed closed, as asserted. By connectedness, $I$ coincides with the entire interval $[0,1]$. In particular, $1\in I$, so that, for all $t\in[0,1]$,
\begin{equation*}
(1-t)\phi_0 + t\phi_1 \in \mathcal{U}_\lambda.
\end{equation*}
That is, $\mathcal{U}_\lambda$ is convex, as desired.
\end{proof}

\bibliographystyle{alpha}
\bibliography{references_articles}
\end{document}